\documentclass[11pt]{amsart}
\usepackage{amsmath, amsfonts, amsthm, amssymb}
\usepackage[longnamesfirst]{natbib}
\usepackage{verbatim}
\usepackage{hyperref}
\usepackage{color}
\usepackage[dvips]{graphicx}

\usepackage[margin, draft]{fixme}
\usepackage{graphicx}
\usepackage{enumerate}         
\newtheorem{thm}{Theorem}[section]
\newtheorem{cor}[thm]{Corollary}
\newtheorem{lem}[thm]{Lemma}
\newtheorem{prop}[thm]{Proposition}
\theoremstyle{definition}
\newtheorem{defn}[thm]{Definition}

\theoremstyle{remark}
\newtheorem{rem}[thm]{Remark}
\newtheorem{example}[thm]{Example}

\numberwithin{equation}{section}

\newcommand{\Xt}{(X_t)_{t\geq0}}

\newcommand{\norm}[1]{\left\Vert#1\right\Vert}

\newcommand{\set}[1]{\left\{#1\right\}}
\newcommand{\Ind}[1]{\mathbf{1}_{\left\{#1\right\}}}
\newcommand{\RR}{\mathbb{R}}

\newcommand{\NN}{\mathbb{N}}
\newcommand{\EE}{\mathbb{E}}
\newcommand{\PP}{\mathbb{P}}
\newcommand{\E}[1]{\mathbb{E}\left[#1\right]}                     
\newcommand{\Ex}[2]{\mathbb{E}^{#1}\left[#2\right]}                     

\newcommand{\cA}{\mathcal{A}}

\newcommand{\Rplus}{\mathbb{R}_{\geqslant 0}}

\newcommand{\pd}[2]{\frac{\partial #1}{\partial #2}}

\newcommand{\supp}{\mathrm{supp}}
\newcommand{\wt}[1]{{\widetilde{#1}}}
\newcommand{\wh}[1]{{\widehat{#1}}}
\newcommand{\ol}{\overline}
\newcommand{\ee}{\mathrm{e}}
\newcommand{\dd}{\mathrm{d}}
\newcommand{\SDplus}{\mathrm{SD}_+}
\newcommand{\CLIM}{\mathrm{CLIM}}

\begin{document}
\title[Limit Distributions of Continuous-State Branching Processes with Immigration]{On the
Limit Distributions of Continuous-State Branching Processes with Immigration}

\author{Martin Keller-Ressel}
\address{Department of Mathematics, TU Berlin, Germany}
\email{mkeller@math.tu-berlin.de}

\author{Aleksandar Mijatovi\'{c}}
\address{Department of Statistics, University of Warwick, UK}
\email{a.mijatovic@warwick.ac.uk}



\begin{abstract}
We consider the class of continuous-state branching processes with immigration (CBI-processes), introduced by~\citet{Kawazu1971} and their limit distributions as time tends to infinity. We determine the L\'evy-Khintchine triplet of the limit distribution and give an explicit description in terms of the characteristic triplet of the L\'evy subordinator and the scale function of the spectrally positive L\'evy process, which describe the immigration resp.\  branching mechanism of the CBI-process. This representation allows us to describe the support of the limit distribution and characterise its absolute continuity and asymptotic behavior at the boundary of the support, generalizing several known results on self-decomposable distributions.
\end{abstract}

\thanks{AM would like to thank the Mathematisches Forschungsinstitut Oberwolfach,
where a part of the work on the paper was carried out. 
Both authors would like to thank Zenghu Li for valuable comments.
We are grateful to an anonymous referee whose suggestions greatly
improved the paper.}

\maketitle

\section{Introduction}

Continuous-state branching processes with immigration
(CBI-processes) have been introduced by~\citet{Kawazu1971} as
scaling limits of discrete single-type branching processes with
immigration. In~\cite{Kawazu1971} the authors show that in general a
CBI-process has a representation in terms of the Laplace exponents
$F(u)$ and $R(u)$ of two independent L\'evy processes: a L\'evy
subordinator $X^F$ and a spectrally positive L\'evy process $X^R$,
which can be interpreted as immigration and branching mechanism of
the CBI-process respectively. Following \citet{Dawson2006} a
CBI-process can in fact be represented as the unique strong solution of a
non-linear SDE driven by $X^F$ and $X^R$.

For discrete branching processes with immigration, limit
distributions have been studied already in \citet{Heathcote1965,
Heathcote1966} and some results on the existence of limit
distributions of a CBI-process were published by \citet{Pinsky1972},
albeit without proofs. Recently, proofs for the results of
\citeauthor{Pinsky1972} have appeared in \citet{Li2011}. The main result
of \citet{Pinsky1972} states that under an integral condition on the
ratio $F(u)/R(u)$ a limit distribution exists and can be described in terms of its Laplace exponent (cf.
Theorem~\ref{Thm:limit_dis}). The contribution of this article is to
build on the results of \citet{Pinsky1972} in order to give a finer
description of the limit distribution: We show that it is infinitely
divisible, give a representation of its L\'evy-Khintchine triplet
(Theorem~\ref{thm:Main}) and then use this new representation to
obtain results on smoothness, support and other properties of the
limit distribution. From this main result, several other representations of the 
L\'evy-Khintchine triplet are then derived. The most concise representation is given by 
equation~\eqref{Prop:AW}, which states that the L\'evy measure of the limit
distribution has a density of the form $x\mapsto k(x)/x$
and the corresponding
$k$-function
$k:(0,\infty)\to\Rplus$
is given
by the formula
$$
k=-\cA_{\breve X^F}W,
$$
where $\cA_{ \breve X^F}$ is the generator of the modified L\'evy
subordinator $\breve X^F$ and $W$ is the \emph{scale function} that
corresponds to the spectrally one-sided L\'evy process $X^R$ (see Remark~\ref{Rem:generator} for the precise statement of this
factorization). In Section~\ref{sec:Properties} we derive further properties of the limit distribution. In particular, we
characterize the support of the limit distribution, show its
absolute continuity and describe its boundary behavior at the left endpoint of the support. Furthermore
we prove that the class of limit distributions of CBI-processes is
strictly larger than the class of self-decomposable distributions on
$\Rplus$ and is strictly contained in the
class of all infinitely divisible distributions on $\Rplus$.

Most of our results can also be regarded as extensions of
known results on limit distributions of Ornstein-Uhlenbeck-type (OU-type), see e.g. \citet{Jurek1983, Sato1984,  Sato1999}, to the class of CBI-processes. The knowledge of the L\'evy-Khintchine triplet for stationary distributions of OU-type processes has been applied to statistical estimation of the underlying process in \citet{Masuda2005}). We suggest that in further research our results may be used for extensions of this methodology to CBI-processes.

\section{Preliminaries}
\label{sec:Preliminaries}
\subsection{Continuous-State Branching Processes with Immigration}\label{Sec:CBI}

Let $X$ be a continuous-state branching process with
immigration.
Following~\citet{Kawazu1971}, such a process is defined as a stochastically continuous Markov
process with state space $[0,\infty]$, whose Laplace exponent is affine in the state variable, i.e. there exist functions $\phi(t,u)$ and $\psi(t,u)$ such that
\begin{equation}\label{Eq:CBI_def}
-\log \Ex{x}{\ee^{-uX_t}} = \phi(t,u) + x \psi(t,u), \qquad \text{for all}\quad t \ge 0, u \ge 0, x \in \Rplus\;,
\end{equation}
where, as usual for the theory of Markov processes,
$\mathbb{E}^x$ denotes expectation, conditional on $X_0 = x$. Since we are interested in the limit behavior of the process
$X$
as $t \uparrow \infty$,
we further assume that $X$ is \emph{conservative}, i.e. that $X_t$ is a proper random variable for each $t \ge 0$
with state space
$\Rplus := [0,\infty)$.
The following theorem is proved in \citet{Kawazu1971}.

\begin{thm}[\citet{Kawazu1971}]\label{Thm:KW}
Let $\Xt$ be a conservative CBI-process.
Then the functions $\phi(t,u)$ and $\psi(t,u)$ in~\eqref{Eq:CBI_def} are differentiable in
$t$ with derivatives
\begin{align}
F(u) = \left.\pd{}{t}\phi(t,u)\right|_{t = 0}, \qquad R(u) =
\left.\pd{}{t}\psi(t,u)\right|_{t = 0}\;
\end{align}
and $F$, $R$ are of Levy-Khintchine form
\begin{gather}
F(u) = bu - \int_{(0,\infty)}{\left(\ee^{-u\xi} - 1\right)\,m(\dd\xi)}\label{Eq:F_def},\\
R(u) = -\alpha u^2 + \beta u - \int_{(0,\infty)}{\left(\ee^{-u\xi} - 1 + u\xi I_{(0,1]}(\xi)
\right)\,\mu(\dd\xi)},\label{Eq:R_def}
\end{gather}
where $\alpha, b \in \Rplus$, $\beta \in \RR$, 
$I_{(0,1]}$
is the indicator function of the interval
$(0,1]$
and $m, \mu$ are
L\'{e}vy measures on $(0,\infty)$, with $m$ satisfying $
\int_{(0,\infty)}{(x \land 1)\,m(\dd x)} < \infty$, and $R$ satisfying\footnote{The notation $\int_{0+}$ denotes an integral over an arbitrarily small right neighborhood of $0$.}
\begin{equation}\label{Eq:conservative}
\int_{0+} \frac{1}{R^*(s)}\dd s = \infty \quad \text{where} \quad R^*(u) = \max(R(u),0).
\end{equation}
Moreover $\phi(t,u)$, $\psi(t,u)$ take values in $\Rplus$ and satisfy the ordinary differential equations
\begin{equation}\label{Eq:Riccati_general}
\begin{split}
\pd{}{t}\phi(t,u) &= F\left(\psi(t,u)\right), \qquad \phi(0,u)
= 0\;,\\
\pd{}{t}\,\psi(t,u) &= R\left(\psi(t,u)\right), \qquad \psi(0,u) =
u\;.
\end{split}
\end{equation}
\end{thm}
\begin{rem}
The equations \eqref{Eq:Riccati_general} are often called generalized Riccati
equations, since they are classical Riccati differential equations, when $m = \mu = 0$.
\end{rem}

We call $(F,R)$ the \emph{functional characteristics} of the CBI-process $X$.
Furthermore
the article of \citet{Kawazu1971} contains the following converse result:
for any functions $F$ and $R$
defined by~\eqref{Eq:F_def}
and~\eqref{Eq:R_def} respectively,
which satisfy
condition~\eqref{Eq:conservative}
and the restrictions
on the parameters
$\alpha,b$,
$\beta$
and the L\'evy measure
$m$
stated in
Theorem~\ref{Thm:KW},
there exists a unique conservative CBI-process with functional characteristics
$(F,R)$.
In this sense the pair $(F,R)$ truly characterizes the process $X$.
Clearly, $F(u)$ is the Laplace exponent of a L\'evy subordinator $X^F$,
and $R(u)$ is the Laplace exponent of a L\'evy process $X^R$ without
negative jumps. Thus, we also have a one-to-one correspondence between
(conservative) CBI-processes and pairs of L\'evy processes $(X^F, X^R)$,
of which the first is a subordinator, and the second a process without negative
jumps that satisfies condition~\eqref{Eq:conservative}. In the case of a
CBI-process without immigration (i.e. a CB-process), which corresponds to
$F = 0$, a pathwise transformation of $X^R$ to $X$ and vice versa was given by~\citet{Lamperti1967},
and is often referred to as `Lamperti transform'. Recently, a pathwise correspondence
between the pair $(X^F, X^R)$ and the CBI-process $X$ has been constructed by \citet{Caballero2010}.

The following properties of $F(u)$ and $R(u)$ can be easily derived from the representations \eqref{Eq:F_def} and \eqref{Eq:R_def} and the parameter conditions stated in Theorem~\ref{Thm:KW}.
\begin{lem}\label{Lem:FR_properties}
The functions $F(u)$ and $R(u)$ are concave and continuous on $\Rplus$ and
infinitely differentiable in $(0,\infty)$. At $u=0$ they satisfy $F(0) = R(0) = 0$, the
right derivatives $F'_+(0)$ and $R'_+(0)$ exist in $(-\infty, +\infty]$ and satisfy $F'_+(0) = \lim_{u \downarrow 0} F'(u)$ and $R'_+(0) = \lim_{u \downarrow 0} R'(u)$.
\end{lem}
We will also need the following result, which can be found e.g. in
\citet[Ch.~8.1]{Kyprianou2006}
\begin{lem}\label{Lem:R_zeroes}
For the function $R(u)$ exactly one of the following holds:
\begin{enumerate}[(i)]
\item $R'_+(0) > 0$ and there exists a $u_0 > 0$ such that $R(u_0) =
0$;\label{Item:supercritical}
\item $R \equiv 0$;\label{Item:degenerate_branching}
\item $R'_+(0) \le 0$ and $R(u) < 0$ for all $u >
0$.\label{Item:subcritical}
\end{enumerate}
\end{lem}
\begin{rem}
 In case \eqref{Item:supercritical} $R(u)$ is called a
supercritical branching mechanism, while case
\eqref{Item:subcritical} can be further distinguished into critical
($R'_+(0) = 0$) and subcritical branching ($R'_+(0) < 0$).
\end{rem}
In what follows we will be interested in the limit distribution and
the invariant distribution of $\Xt$. We write
$P_tf(x)=\EE^x[f(X_t)]$ for all $x\in\Rplus$ and denote by
$(P_t)_{t\geq 0}$ the transition semigroup associated to the Markov
process $X$. We say that $L$ is the \textit{limit distribution} of
the process $X=\Xt$ if $X_t$ converges in distribution to $L$ under
all $\PP^x$ for any starting value $x \in \Rplus$ of $X$. We call
$L$ an \textit{invariant (or stationary) distribution} of $X=\Xt$,
if
\[\int_{[0,\infty)} P_t f(x) \dd L(x) = \int_{[0,\infty)}f(x) \dd L(x),\]
for any $t \ge 0$ and bounded measurable $f: \Rplus \to \Rplus$.
Finally we denote the Laplace exponent of $L$ by
\[l(u) = -\log \int_{[0,\infty)} e^{-ux} \dd L(x), \qquad (u \ge 0).\]
\subsection{Limit Distributions of CBI-Processes}
Theorem~\ref{Thm:limit_dis} and Corollary~\ref{Cor:log_moment}
concern the existence of a limit distribution of a CBI-process and
have been announced in a similar form but without proof in
\citet{Pinsky1972}. A proof has recently appeared in
\citet[Thm.~3.20, Cor~3.21]{Li2011}; the only difference to the result given here is
that we drop a mild moment condition assumed in
\citet[Eq.(3.1)f]{Li2011} and that we include stationary
distributions in the statement of our result. Some weaker results on the
existence of a limit distribution of a CBI-process have also
appeared in~\citet{KS2008}. We give a self-contained proof of the
theorem and its corollary in the appendix of the article.

\begin{thm}[\citet{Pinsky1972, Li2011}]\label{Thm:limit_dis}
Let $\Xt$ be a CBI-process on $\Rplus$. Then the following statements
are equivalent:
\begin{enumerate}[(a)]
\item $\Xt$ converges to a limit distribution $L$ as $t \to
\infty$;\label{Item:limit}
\item $\Xt$ has the unique invariant distribution $L$;\label{Item:invariant}
\item It holds that $R'_+(0) \le 0$ and \label{Item:integral_condition}
\begin{equation}\label{Eq:integral_condition}
-\int_0^u{\frac{F(s)}{R(s)} \dd s} < \infty
\end{equation}
for some $u > 0$.
\end{enumerate}
Moreover the limit distribution $L$ has the following properties:
\begin{enumerate}[(i)]
\item $L$ is infinitely divisible;\label{Item:id}
\item the Laplace exponent $l(u) = -\log \int_{[0,\infty)} e^{-ux} \dd L(x)$ of $L$ is given by
\begin{equation}\label{Eq:Gen_Jurek_Vervaat}
l(u) = -\int_0^u{\frac{F(s)}{R(s)} \dd s} \qquad (u \ge 0)\;.
\end{equation}
\label{Item:laplace}
\end{enumerate}
\end{thm}
\begin{rem}
Note that the existence of the right derivative $R'_+$ at $0$ that
appears in statement \eqref{Item:integral_condition} is guaranteed
by Lemma~\ref{Lem:FR_properties}.
\end{rem}

\begin{cor}\label{Cor:log_moment} If $R'_+(0) < 0$ then the integral condition \eqref{Eq:integral_condition} is equivalent to
the log-moment condition
\begin{equation}\label{Eq:logmoment}
\int_{\xi > 1}{\log \xi\,m(\dd\xi)} < \infty.
\end{equation}\label{Item:logmoment}
\end{cor}

\subsection{Results on Ornstein-Uhlenbeck-type Processes}\label{Sec:OU-processes}
A subclass of CBI-processes, whose limit distributions have been
studied extensively in the literature is the class of
$\Rplus$-valued Ornstein-Uhlenbeck-type (OU-type) processes. We briefly discuss some of the known results on 
OU-type processes, that will be generalized by our results in the next section. 
Let $\lambda > 0$ and $Z$ be a L\'evy subordinator with drift
$b\in\Rplus$ and L\'evy measure $m(\dd \xi)$. 
An
$\Rplus$-valued OU-type process $X$ is the strong solution of the SDE
\begin{equation}\label{Eq:OU_SDE}
dX_t = - \lambda X_t dt + dZ_t,\qquad X_0 \in \Rplus,
\end{equation}
which is given by $X_t=X_0e^{-\lambda t}+\int_0^t e^{\lambda(s-t)}dZ_s$.
This is the classical
Ornstein-Uhlenbeck process, where the Brownian motion has been replaced by
an increasing L\'evy process. It follows from elementary
calculations that an $\Rplus$-valued OU-type process is a
CBI-process with $R(u) = - \lambda u$. In terms of the two L\'evy
processes $X^F$, $X^R$, this corresponds to the case that $X^F = Z$,
and $X^R$ is the degenerate L\'evy process $X^R_t = -\lambda t$. For
OU-type processes analogues of Theorem~\ref{Thm:limit_dis} and the
log-moment condition of Corollary~\ref{Cor:log_moment} already
appeared in \citet{Cinlar1971}.

An interesting characterization of the limit distributions of
OU-type processes is given in terms of self-decomposability: Recall
that a random variable $Y$ has a \textit{self-decomposable
distribution} if for every $c \in [0,1]$ there exists a random
variable $Y_c$, independent of $Y$, such that
\begin{equation}\label{Eq:SD_def}
Y \stackrel{d}{=} cY + Y_c.
\end{equation}
Self-decomposable distributions are a subclass of infinitely divisible distributions,
and exhibit in many aspects an increased degree of regularity. It is known for example,
that every non-degenerate self-decomposable distribution is absolutely continuous (cf. \citet[27.8]{Sato1999}) and
unimodal (cf. \citet{Yamazato1978} or \citet[Chapter~53]{Sato1999}),
neither of which holds for general infinitely divisible distributions.
As we are working with non-negative processes, we focus on self-decomposable distributions on the
half-line $\Rplus$, and we denote this class by $\SDplus$.
The connection to OU-type processes is made by the following result.

\begin{thm}[\citet{Jurek1983, Sato1984}]\label{Thm:Jurek}
Let $X$ be an OU-type process on $\Rplus$ and suppose that $m(\dd\xi)$
satisfies the log-moment condition $\int_{\xi > 1}{\log
\xi\,m(\dd\xi)} < \infty$. Then $X$ converges to a limit distribution $L$ which is self-decomposable.
Conversely, for every self-decomposable distribution $L$ with
support $\Rplus$ there exists a unique subordinator $Z$ with drift $b\in\Rplus$
and a L\'evy measure $m(\dd\xi)$, satisfying
$\int_{\xi > 1}{\log \xi\,m(\dd\xi)} < \infty$, such that
$L$ is obtained as the limit distribution of the corresponding OU-type process.
\end{thm}

Since a self-decomposable distribution is infinitely divisible, its
Laplace exponent has a L\'evy-Khintchine decomposition. The
following characterization is due to Paul L\'evy and can be found in~\citet[Cor.~15.11]{Sato1999}:  
an infinitely divisible distribution $L$ on $\Rplus$ is self-decomposable, if and only if its Laplace
exponent is of the form
\begin{equation}\label{Eq:kfunction}
-\log \int_{[0,\infty)} e^{-ux} \dd L(x) = \gamma u - \int_0^\infty{\left(\ee^{-ux} - 1\right)\frac{k(x)}{x} \dd x},
\end{equation}
where $\gamma \ge 0$ and $k$ is a decreasing function on $\Rplus$. The parameters $\gamma$ and $k$ are related to the L\'evy subordinator $Z$ by
\begin{equation}\label{Eq:krepresentation}
\gamma = \frac{b}{\lambda} \qquad \text{and} \qquad k(x) = \frac{1}{\lambda}m(x,\infty) \quad \text{for} \quad x > 0.
\end{equation}
Following~\cite{Sato1999} we call $k$ the $k$-function of the
self-decomposable distribution $L$. Many properties of $L$, such as
smoothness of its density, can be characterized through $k$. In
fact, several subclasses of $\SDplus$ have been defined, based on
more restrictive assumptions on $k$. For example, the class of
self-decomposable distributions whose $k$-function is completely
monotone, is known as the Thorin class, and arises in the study of
mixtures of Gamma distributions; see \citet{James2008} for an
excellent survey. In our main result, Theorem~\ref{thm:Main} we give
analogues of the formulas \eqref{Eq:kfunction} and
\eqref{Eq:krepresentation} for the limit distribution of a
CBI-process. As it turns out, a representation as in
\eqref{Eq:kfunction} still holds, with the class of decreasing
$k$-functions replaced by a more general family. However, we do not obtain a structural characterization of the CBI limit distributions that replaces self-decomposability. Identifying such a structural condition (if there is any) constitutes an interesting question that is left open by our results.

\section{L\'evy-Khintchine Decomposition of the Limit Distribution}\label{Sec:main}

Let $X^F$ and $X^R$ be the L\'evy processes that correspond to the Laplace exponents $F$ and $R$
given in~\eqref{Eq:F_def} and~\eqref{Eq:R_def} respectively. As remarked in Section~\ref{Sec:CBI}, $X^F$ is a subordinator and
 $X^R$ is a L\'evy process with no negative jumps. From Theorem~\ref{Thm:limit_dis} it follows that whenever a limit distribution exists,
 then $\EE\left[X^R_1\right]=R_+'(0) \le 0$ and $R \not \equiv 0$, such that $X^R$ is not a subordinator, but a true
 \emph{spectrally positive L\'evy process} in the sense of \citet{Bertoin1996}. The fluctuation theory of
spectrally one-sided L\'evy processes has been studied extensively.
The convention used in much of the literature is to study a
spectrally negative process. In our setting such a process is given
by the dual $\wh{X}^R=-X^R$ and its Laplace exponent is
$\log \E{\ee^{u\wh{X}^R}} = -R(u)$ for $u\geq0$.
A central result in the fluctuation theory of spectrally one-sided L\'evy processes
(see~\cite[Thm~8,~Ch~VII]{Bertoin1996}) states that for each function
$R$ of the form~\eqref{Eq:R_def}
there exists a unique function 
$W:\RR\to[0,\infty)$,
known as the \emph{scale function} of $\wh{X}^R$,
which is increasing and continuous on the interval 
$[0,\infty)$
with Laplace transform
\begin{equation}\label{Eq:Laplace}
\int_0^\infty{\ee^{-ux}W(x)\dd x} = -\frac{1}{R(u)}, \quad \text{for}\quad u > 0.
\end{equation}
and identically zero on 
the negative half-line ($W(x)=0$
for all
$x<0$).
Note that this equality implies that
\begin{equation}\label{Eq:W_little_o}
W(x) = o(e^{\epsilon x}) \quad \text{as} \quad x \to \infty \quad
\text{for any} \quad \epsilon > 0,
\end{equation}i.e. that $W$ has
sub-exponential growth, a fact that will be needed subsequently. 
Furthermore the scale function $W$ has the
representation
\begin{eqnarray}
\label{eq:Scale_Function_Compensted_ReP}
\frac{W(x)}{W(y)} & = & \exp\left\{-\int_x^y n\left(\ol\varepsilon\geq z\right)\,\dd
z\right\}\qquad\text{for any}\qquad 0<x<y,
\end{eqnarray}
where
$n$
is
the It\^o excursion measure on the set
\begin{equation}
\label{eq:ExcursionSet}
\mathcal E=\left\{\varepsilon\in D(\RR):\exists \zeta_\varepsilon\in(0,\infty]\text{ s.t.
}\varepsilon(t)=0
\text{ if } \zeta_\varepsilon\leq t<\infty,\>
\varepsilon(0)\geq0,\>\varepsilon(t)>0\>\forall t\in(0,\zeta_\varepsilon)\right\},
\end{equation}
with
$D(\RR)$
the Skorokhod space. The measure $n$ is the intensity measure of
the Poisson point process of excursions
from the supremum
of
$\wh X^R$
and
$\{\ol \varepsilon\geq z\}\subset \mathcal E$
denotes the set of excursions of height
$\ol \varepsilon=\sup_{t< \zeta_\varepsilon}\varepsilon(t)$
at least
$z>0$
(see~\citet{Bertoin1996} for details on the It\^o excursion theory in
the context of L\'evy processes).
The representation ~\eqref{eq:Scale_Function_Compensted_ReP}
implies that, 
on the interval $(0,\infty)$, the scale function
$W$ is strictly positive, absolutely continuous,
log-concave with 
right-
and left-derivative 
given by
$W_+'(x) = n\left(\ol\varepsilon> x\right) W(x)$ and
$W_-'(x) = n\left(\ol\varepsilon\geq x\right) W(x)$
respectively.
Furthermore 
at
$x=0$
the right-derivative
$W_+'(0)$
exists in
$[0,\infty]$.

Using the scale function $W$ associated to $\wh{X}^R$ we can
formulate our main result on the L\'evy-Khintchine decomposition of
the limit distribution of a CBI-process.

\begin{thm}
\label{thm:Main} Let $X$ be a CBI-process with functional
characteristics $(F,R)$ given in~\eqref{Eq:F_def}
and~\eqref{Eq:R_def} and assume that $X$ converges
to a limit distribution $L$. Let $W$ be the scale function
associated to the dual $\wh X^R$ of the spectrally positive L\'evy
process $X^R$
and let $(b,m)$ be the drift and L\'evy measure of the
subordinator
$X^F$.
Then
$L$ is infinitely divisible, and its Laplace exponent has the
L\'evy-Khintchine decomposition
\begin{equation}\label{Eq:LK_limit}
-\log \int_0^\infty e^{-ux} \dd L(x) = u\gamma -
\int_{(0,\infty)}{\left(\ee^{-xu} - 1\right)\frac{k(x)}{x}}\, \dd x,
\end{equation}
where $\gamma \ge 0$ and $k: (0,\infty) \to \Rplus$ are given by
\begin{align}
\label{eq:MainRepresentation_gamma}
\gamma &= bW(0),\\
k(x) &= b\, W_+'(x) + 
\int_{(0,\infty)}\left[W(x)-W(x-\xi)\right]\,m(\dd \xi). \label{eq:MainRepresentation_k}
\end{align}
\end{thm}

\begin{rem}
If $X$ is an OU-type process, then $R$
is of the form
$R(u)=-\lambda u$
with
$\lambda>0$. 
Since the scale function in this case takes the form
$W(x)=\frac{1}{\lambda}I_{[0,\infty)}(x)$,
equations \eqref{eq:MainRepresentation_gamma} 
and~\eqref{eq:MainRepresentation_k}
reduce to $\gamma = \tfrac{b}{\lambda}$ and $k(x) = \tfrac{1}{\lambda} m(x,\infty)$.
This is precisely the known result for the $\Rplus$-valued OU-type process stated in~\eqref{Eq:krepresentation}.
\end{rem}

Before proving this result, 
we state a corollary that connects the limit distribution in 
Theorem~\ref{thm:Main}
with the excursion measure 
$n$ associated to the Poisson point process of
excursions away from the supremum of the dual of the branching mechanism
$X^R$.
To state it, we introduce the effective drift $\lambda_0$ of
$\wh{X}^R$, which is defined as
\begin{equation}\label{Eq:effective_drift}
\lambda_0 = \begin{cases} \int_{(0,1]} \xi \mu(\dd \xi)-\beta,&\quad \text{if} \; X^R \;
\text{has bounded variation,}\\
+\infty, &\quad \text{if} \; X^R\; \text{has unbounded
variation.}\end{cases}
\end{equation}
Note that $\lambda_0>0$ must hold if $R'_+(0) \le 0$.

\begin{cor}\label{cor:ProbInterp_0}
Let the assumptions of Theorem~\ref{thm:Main}
hold. Then
\begin{align}
\label{eq:gamma_via_beta}
\gamma  &= \frac{b}{\lambda_0}\qquad\text{and}\\
k(x)\  &=  W(x) \left(b\,n( \ol{\varepsilon} > x) +
\int_{(0,\infty)}\left(1-\Ind{\xi \le x}\exp\left(-\int_{x-\xi}^x n(\ol{\varepsilon}‚ \ge z)\dd z\right)\right)\,m(\dd \xi)\right)
\label{eq:k_via_n}
\end{align}
where $\lambda_0$ is the effective drift of $\wh{X}^R$,
$n$
is
the It\^o excursion measure corresponding to
the Poisson point process of excursions from the supremum of $\wh{X}^R$.
\end{cor}

\begin{proof}[Proof of Theorem~\ref{thm:Main}]
For every $t>0$, the distribution of $X_t$ is infinitely divisible
and supported on $\Rplus$. Hence the same is true of the limit
distribution $L$. The Laplace exponent of $L$ can by
Theorem~\ref{Thm:limit_dis} be expressed as
\begin{equation}
\label{Eq:limit_LK} -\log \int_{[0,\infty)} e^{-ux} \dd L(x)  =
-\int_0^u{\frac{F(s)}{R(s)}ds} = du -
\int_{(0,\infty)}{\left(\ee^{-ux} - 1 \right)\nu(\dd x)},\qquad u
\ge 0,
\end{equation}
for $d \ge 0$ and some L\'evy measure $\nu(\dd x)$ satisfying
$\int_{(0,1)}x\nu(\dd x) < \infty$. Moreover, it is clear from
Theorem~\ref{Thm:limit_dis} that $R_+'(0) \le 0$ and $R \not \equiv
0$. Thus, Lemma~\ref{Lem:R_zeroes} implies that the quotient $F/R$
is continuous at any $u>0$. Since the elementary inequality
$|\ee^{-xh}-1|/h<x$ holds for all $x,h>0$, the dominated convergence
theorem and the fundamental theorem of calculus applied
to~\eqref{Eq:limit_LK} yield the identity
\begin{eqnarray}
\label{Eq:limit_LK_diff}
 - \frac{F(u)}{R(u)} & = & d + \int_{(0,\infty)}{\ee^{-ux} x \nu(\dd x)} \qquad\text{for all}\qquad u>0.
\end{eqnarray}

Any twice-differentiable function
$f$
that tends to zero
as
$|x|\to\infty$,
i.e.
$f\in C_0^2(\RR)$,
is in the domain of the
generator $\cA_{X^F}$
of the subordinator
$X^F$
and the following formula holds
\begin{eqnarray}
\label{eq:Gen_X_F}
\cA_{X^F} f(x)  & =  & b f'(x) + \int_{(0,\infty)}\left[f(x + \xi) -
f(x)\right]\,m(\dd\xi)\qquad\text{for}\qquad x\in\RR.
\end{eqnarray}
Fix $u>0$ and let $f_u\in C_0^2(\RR)$ be a function that satisfies
$f_u(x) = \ee^{-ux}$ for all $x\geq0$. Applying \eqref{eq:Gen_X_F}
to $f_u$ yields $\cA_{X^F} f_u(x) = F(u) f_u(x)$ for all $u > 0, x
\ge 0$. Multiplying by $-W(x)$ and integrating from $0$ to $\infty$
gives the following identity for all $u>0$:
\begin{eqnarray}
\label{eq:First_Step_in_Proof}
\lefteqn{-\int_0^\infty W(x)\cA_{X^F} f_u(x)\,\dd x} \\
&= & b \int_0^\infty u\ee^{-ux} W(x)\,\dd x -
\int_0^\infty W(x)\, \int_{(0,\infty)}\left[\ee^{-u(x+\xi)}-\ee^{-ux}\right]\,m(\dd
\xi)\dd x.\nonumber
\end{eqnarray}
Note that $\cA_{X^F} f_u(x) \sim e^{-ux}$ for $x \to \infty$, which
guarantees that the integrals are finite in light of
eq.~\eqref{Eq:W_little_o}. Since $W$ is increasing and absolutely
continuous, integration by parts gives 
\begin{eqnarray}
\label{eq:Main_Representation_of_Proof}
\lefteqn{-\int_0^\infty\!\! W(x)\cA_{X^F} f_u(x)\,\dd x
 =  b W(0) + b \int_{0}^\infty \ee^{-ux}W_+'(x)\,\dd x}\\
&+ &   \int_0^\infty\ee^{-u x} \int_{(0,\infty)}\left[W(x)-W(x-\xi)\right]\,
m(\dd \xi)\,\dd x,
\nonumber
\end{eqnarray}
for all $u > 0$. The second integral on the right-hand side of~\eqref{eq:Main_Representation_of_Proof}
is a consequence of the following steps: (i)
note that the corresponding integrand in~\eqref{eq:First_Step_in_Proof}
does not change sign on the domain of integration,
(ii) approximate the L\'evy measure
$m$
by a sequence of measures
$(m_n)_{n\in\NN}$
with finite mass, (iii) apply Fubini's theorem to obtain the formula
for each
$m_n$, (iv) take the limit by applying the monotone convergence theorem.

On the other hand, combining the identity $\cA_{X^F}f_u(x) = F(u)
e^{-ux}$ for all $u > 0, x \in \Rplus$ with \eqref{Eq:Laplace} and
\eqref{Eq:limit_LK} yields
\begin{equation}
\label{eq:Next_Representation}
-\int_0^\infty W(x) \cA_{X^F} f_u(x) \dd x = -\int_0^\infty W(x) F(u) e^{-ux} \dd x = -\frac{F(u)}{R(u)} = d + \int_{(0,\infty)}e^{-ux}x \nu(\dd x),
\end{equation}
which in turn must equal the right hand side of~\eqref{eq:Main_Representation_of_Proof}. 
In the limit as
$u \to \infty$, 
the equality of the expressions in~\eqref{eq:Main_Representation_of_Proof}
and~\eqref{eq:Next_Representation}
yields
$d = bW(0)$. 
Subtracting this term we arrive at
the equality
\begin{equation*}
\int_0^\infty\ee^{-u x} \left[bW'_+(x) + \int_{(0,\infty)}\left[W(x)-W(x-\xi)\right]\,
m(\dd \xi)\right]\,\dd x = \int_{(0,\infty)}e^{-ux}x \nu(\dd x).
\end{equation*}
Both sides are Laplace transforms of Borel measures on $(0,\infty)$, and we conclude from the equality of transforms the equality of the measures
\[\left[bW'_+(x) + \int_{(0,\infty)}\left[W(x)-W(x-\xi)\right]\,
m(\dd \xi)\right]\,\dd x = x \nu(\dd x)\]
for all $x > 0$. In particular it follows that $\nu(\dd x)$ has a density $k(x)/x$ with respect to the Lebesgue measure, and that $k(x)$ is given by
\[k(x) = bW'_+(x) + W(x) m(x,\infty)+\int_{(0,x]}\left[W(x)-W(x-\xi)\right]\]
almost everywhere, which concludes the proof.
\end{proof}

\begin{proof}[Proof of Corollary~\ref{cor:ProbInterp_0}]
Following~\citet[Lemma~8.6]{Kyprianou2006}, $W_+(0) > 0$ if and only if $X^R$
has finite variation, and is equal to $1/\lambda_0$ in this case,
with $\lambda_0$ defined by \eqref{Eq:effective_drift}. If $X^R$ has
infinite variation, then $W(0) = 0$ and $\lambda_0 = \infty$, and
hence equation~\eqref{eq:gamma_via_beta} holds. Substituting the
representation~\eqref{eq:Scale_Function_Compensted_ReP} of $W$ in
terms of the measure $n$ into \eqref{eq:MainRepresentation_k} yields
the second equation~\eqref{eq:k_via_n}. 
\end{proof}

\begin{rem}\label{Rem:generator}
The formula for the $k$-function in~\eqref{eq:MainRepresentation_k} looks very
much like the Feller generator of the subordinator $X^F$ applied to
the scale function $W$ of $\wh{X}^R$. However, 
the Feller generator is only defined on a subset (i.e. its domain)
of the Banach space of continuous functions that tend to 
$0$
at infinity, 
$C_0(\RR)$. Any function in the domain of the generator of $X^F$ must be in $C_0(\RR)$ and differentiable; 
sufficient conditions for the differentiability of 
$W$ are given in~\cite{Chan_et_al2011}.
However, the scale function $W$, which is non-decreasing on 
$\Rplus$,
is not in $C_0(\RR)$ and thus never in the domain of the Feller generator of $X^F$. To remedy this problem, consider that
by~\eqref{Eq:W_little_o}
and~\eqref{eq:Scale_Function_Compensted_ReP}, 
both
$W$
and
$W_+'$
are elements of the weighted $L_1$-space
defined by 
\[L_1^{h}(0,\infty) := \set{f \in L_{1}^\text{loc}(0,\infty): \int_0^\infty|f(x)|h(x)\dd x <
\infty},\]
where
$h: (0,\infty) \to (0,\infty)$
is a continuous bounded function with
$\lim_{x\downarrow0}h(x) = 0$ and $h(x) \sim e^{-cx}$ as $x \to \infty$ 
for some 
$c > 0$.
The semigroup $(\breve P_t)_{t \ge 0}$
of the Markov process
$\breve X^{F,x}_t = (x - X^F_t)I_{\{X^F_t \le x\}}+ \partial I_{\{X^F_t > x\}}$
(i.e. the dual, started at $x>0$, of 
$X^F$,
sent to a killing state
$\partial$
upon the first passage into $(-\infty,0)$)
acts on $L_1^{h}(0,\infty)$ by 
$\breve P_tf(x) = \E{    f \left(\breve X_t^{F,x}\right)} = \E{f \left(x - X^F_t \right)\Ind{X^F_t <
x}},$
for each $f \in L_1^{h}(0,\infty)$
where we take
$f(\partial)=0$.
It can be shown that 
the $L_1^h(0,\infty)$-semigroup
$(\breve P_t)_{t \ge 0}$
is strongly continuous
with a generator $\cA_{\breve X^F}$
and, if $R'_+(0) \le 0$, then the scale function $W$ associated to 
$R$ is in the domain of $\cA_{\breve X^F}$. Furthermore the $k$-function $k$ in
Theorem~\ref{thm:Main} can be written as
\begin{equation}
\label{Prop:AW}
k = -\cA_{\breve X^F} W.
\end{equation}
The proof of these facts is straightforward but technical and rather lengthy and hence omitted. 
\end{rem}

\section{Further Properties of the Limit Distribution}
\label{sec:Properties}

As discussed in Section~\ref{Sec:OU-processes},
self-decomposable distributions, which arise as limit distributions of $\Rplus$-valued OU-type processes
are in many aspects more regular than general infinitely divisible distributions.
For self-decomposable distributions precise results are known about their support,
absolute continuity and behavior at the boundary of their support.
Using the notation of Section~\ref{Sec:OU-processes},
and excluding the degenerate case of a distribution concentrated in a single point,
the following holds true when $L$ is self-decomposable:
\begin{enumerate}[(i)]
\item the support of $L$ is $[b/\lambda,\infty)$;
\item the distribution of $L$ is absolutely continuous;
\item the asymptotic behavior of the density of $L$ at $b/\lambda$ is determined by
$c = \lim_{x \downarrow 0}k(x)$.
\end{enumerate}
We refer the reader to~\citet[Theorems~15.10, 24.10, 27.13 and~53.6]{Sato1999}.
The goal in this section is to show analogous results for the limit
distributions $L$ arising from general CBI-processes, i.e. to
characterize the support, the continuity properties and the asymptotic behavior
at the boundary of the support of $L$. We start by isolating the degenerate cases.
A L\'evy process, such as $X^F$ or $X^R$,
is called degenerate, if is
deterministic, or equivalently if its Laplace exponent is of the
form $u \mapsto \lambda u$ for some $\lambda \in \RR$.
For a CBI-process
$X$
we draw a finer distinction.

\begin{defn}
A CBI-process $X$ is \emph{degenerate of the first kind}, if it
is deterministic for all starting values $X_0=x \in \Rplus$.
$X$
is \emph{degenerate of the second kind}, if it is deterministic when
started at $X_0 = 0$.
\end{defn}

\begin{rem}
\label{ref:II_Second_kind}
Clearly degeneracy of the first kind implies degeneracy of the
second kind. From Theorem~\ref{Thm:KW} the following can be easily
deduced: a CBI-process is degenerate of the first kind if and only
if both $X^F$ and $X^R$ are degenerate. In this case $F(u) = bu$ and
$R(u) = \beta u$, and $X$ is the deterministic process given by
\begin{equation}\label{Eq:degenerate_CBI}
X_t = X_0 e^{\beta t} + \frac{b}{\beta} \left(e^{\beta t}-1\right).
\end{equation}
A CBI-process $X$ is degenerate of the second kind, but not of the first,
if and only if $X^F = 0$ and $X^R$ is non-degenerate. In this case it is a CB-process, i.e. a continuous-state branching process \emph{without} immigration.
\end{rem}

\smallskip

The following proposition describes the support of the limit $L$ in the degenerate cases.

\begin{prop}\label{Prop:degenerate_support} Let $X$ be a CBI-process,
let $L$ be its limit distribution and let $k$ be the function
defined in Theorem~\ref{thm:Main}. If $X$ is degenerate of the first
kind, then $\supp\,L = \set{-b/\beta}$. If $X$ is degenerate of the
second kind but not of the first kind, then $\supp\, L = \set{0}$.
Moreover, the following statements are equivalent:
\begin{enumerate}[(a)]
\item \label{Item:atomic_support} the support of $L$ is concentrated at a single point;
\item \label{Item:degenerate_process} $X$ is degenerate (of either
kind);
\item \label{Item:kzero} there exists a sequence $x_i \downarrow 0$ such that $k(x_i) = 0$ for all $i \in \NN$;
\item \label{Item:kzero2} $k(x) = 0$ for all $x > 0$.
\end{enumerate}
\end{prop}

\begin{proof}
Suppose that $X$ is degenerate of the first kind. Then the limit $L$ is concentrated at $-b/\beta$ by~\eqref{Eq:degenerate_CBI}.
Suppose next, that $X$ is degenerate of the second, but not the first kind. Then $F = 0$ and by
Theorem~\ref{thm:Main} the Laplace exponent of $L$ is $0$.
It follows that $L$ is concentrated at $0$ in this case.

We proceed to show the second part of the proposition.
It is obvious that~\eqref{Item:kzero2} 
implies~\eqref{Item:kzero}.
To show that~\eqref{Item:kzero} implies~\eqref{Item:degenerate_process},
note that the inequality
$$k(x) \ge
W(x) \left(b\,n( \ol{\varepsilon} > x) +m(x,\infty) \right)
$$
holds for
all
$x>0$
by equation~\eqref{eq:k_via_n}.
Since $W(x) > 0$
for any
$x > 0$, assumption~\eqref{Item:kzero}
implies that
$$ bn(\ol{\varepsilon} > x_i)+m(x_i,\infty)=0
\quad\text{for all}\quad x_i, i \in \NN.$$ We can conclude that $m
\equiv 0$ and hence $X^F_t=bt$ for all $t\geq0$. Furthermore we see
that either $b=0$ or $n \equiv 0$. If $b=0$, then $F=0$ and hence,
by Remark~\ref{ref:II_Second_kind}, $X$ is a degenerate CBI-process
of the second kind. On the other hand, if the It\^o excursion
measure $n$ is zero, then the representation
in~\eqref{eq:Scale_Function_Compensted_ReP} implies that the scale
function $W$ is constant. In this case it follows
from~\eqref{Eq:Laplace} that $R(u)=\beta u$ for some $\beta<0$, or
equivalently that $X_t^R = \beta t$ for all $t\geq0$ and hence that
$X$ is degenerate of the first kind.

The fact that \eqref{Item:degenerate_process} implies
\eqref{Item:atomic_support} follows from the first part of the
proposition. It remains to show that \eqref{Item:atomic_support}
implies \eqref{Item:kzero2}; this is a consequence of the fact that
$L$ is infinitely divisible with support in
$\Rplus$, and that the support of an infinitely
divisible distribution in
$\Rplus$
is concentrated at a single point if and only
if its L\'evy measure is trivial (cf. \citet[Thm.~24.3, Cor.~24.4]{Sato1999}).
\end{proof}

The next result describes the support of the limit $L$ in the non-degenerate case.

\begin{prop}\label{Prop:support}
Let $X$ be a non-degenerate CBI-process
and let $L$ be its limit distribution. Then
\[\supp\,L = [b/\lambda_0, \infty),\]
where $\lambda_0$ is the effective drift of $\wh{X}^R$, defined in \eqref{Eq:effective_drift}.
In particular $\supp\,L = \Rplus$ if and only if $b = 0$ or the paths of $X^R$ have
infinite variation.
\end{prop}

\begin{proof}
From Proposition~\ref{Prop:degenerate_support}\ref{Item:kzero} we
know that there is some $\delta > 0$, such that the $k$-function of
$L$ is non-zero on $(0,\delta)$. For any $h \in (0,\delta)$, define
$L_h$ as the infinitely divisible distribution with Laplace exponent
$\int_{(h,\infty)} \left(e^{-xu} - 1\right) \tfrac{k(x)}{x} \dd x$.
Each $L_h$ is a compound Poisson distribution, with L\'evy measure
$\nu_h(\dd \xi) = \tfrac{k(x)}{x} \mathbf{1}_{(h,\infty)}(x)$. Since
$k$ is non-zero on $(0,\delta)$
\begin{equation} \label{Eq:support_ineq}
(h,\delta)\; \subset \; \supp\,\nu_h \;\subset\; (h,\infty).
\end{equation}
From Theorem~\ref{thm:Main} we deduce that as $h \to 0$ the
distributions $L_h$ converge to $L(\gamma + .)$, i.e. to $L$ shifted
to the left by $\gamma$. For the supports, this implies that
\begin{equation}
\label{Eq:support}
\supp\,L = \set{\gamma}+\overline{\lim_{h \downarrow 0} \supp\,L_h}
\end{equation}
where the limit denotes an increasing union of sets
and
`+' denotes pointwise addition of sets.
Using
\eqref{Eq:support_ineq} and the fact that $L_h$ is a compound
Poisson distribution it follows that
$$\set{0} \cup \bigcup_{n=1}^\infty (n h, n\delta) \; \subset \; \supp\,L_h \; \subset \; \set{0}
\cup(h,\infty), 
$$
by \citet[Thm.~24.5]{Sato1999}.
Let $h \downarrow 0$
and
apply~\eqref{Eq:support} to obtain
$$\bigcup_{n=1}^\infty [\gamma,\gamma + n\delta] \; \subset \; \supp\,L \; \subset \; [\gamma,\infty),$$
and we conclude that $\supp\,L = [\gamma,\infty)$. By Corollary~\ref{cor:ProbInterp_0} $\gamma = b/\lambda_0$, which completes the proof.
\end{proof}

\begin{prop}\label{Cor:absolute_continuity}
Let $X$ be a CBI-process
and let $\lambda_0$ be as in~\eqref{Eq:effective_drift}.
Then the limit distribution $L$ is either
absolutely continuous on $\Rplus$ or absolutely continuous on
$\Rplus \setminus \set{b/\lambda_0}$ with an atom at
$\set{b/\lambda_0}$, according to whether
\begin{equation}\label{Eq:integral_test}
\int_0^1 \frac{k(x)}{x} \dd x = \infty \qquad \text{or} \qquad \int_0^1 \frac{k(x)}{x} \dd x < \infty.
\end{equation}
\end{prop}
\begin{proof}
If $X$ is degenerate, then the assertion follows immediately from
Proposition~\ref{Prop:degenerate_support}. In this case $k(x) = 0$
for all $x > 0$, the integral in \eqref{Eq:integral_test} is always
finite and the distribution of $L$ consists of a single atom at
$b/\lambda_0$.

It remains to treat the non-degenerate case.
Assume first that the integral in \eqref{Eq:integral_test} takes a
finite value. Then also the total mass $\nu(0,\infty)$ of the L\'evy
measure $\nu(\dd x) = \tfrac{k(x)}{x}\dd x$ is finite, and $L -
\gamma$ has compound Poisson distribution. By
\citet[Rem.~27.3]{Sato1999} this implies that for any Borel-set $A \subset \Rplus$
\begin{equation}\label{Eq:compound_poisson}
\int_{A + \gamma} \dd L(x) = e^{-t\nu(0,\infty)} \sum_{j=0}^\infty \frac{t^j}{j!}\nu^{*j}(A),
\end{equation}
where $\nu^{*j}(\dd x)$ is the $j$-th convolution power of $\nu$,
and it is understood that $\nu^{*0}$ is the Dirac measure at $0$.
Since $\nu(\dd x)$ is absolutely continuous -- it has density
$\tfrac{k(x)}{x}$ -- also the convolution powers $\nu^{*j}(\dd x)$
are absolutely continuous for $j \ge 1$. The first summand
$\nu^{*0}$ however has an atom at $0$. It follows by
\eqref{Eq:compound_poisson} that $L - \gamma$ is absolutely
continuous on $(0,\infty)$ with an atom at $0$, and we have shown
the claim for the case $\int_0^1 \frac{k(x)}{x} \dd x < \infty$.

Assume that $\int_0^1 \frac{k(x)}{x} \dd x = \infty$. Then the L\'evy measure $\nu(\dd x) = \tfrac{k(x)}{x} \dd x$ of $L$ has infinite total mass, and \citet[Thm.~27.7]{Sato1999} implies that $L$ has a distribution that is absolutely continuous, which completes the proof.
\end{proof}

So far, we know that the left endpoint of the support of $L$ is
$\gamma = b/\lambda_0$, and that the distribution of $L$ may or may
not have an atom at this point. In case that there is no atom, the
following proposition yields an even finer description of the
behavior of the distribution close to $\gamma$.

\begin{prop}
Let $X$ be a CBI-process satisfying the assumptions of
Theorem~\ref{thm:Main}, and let $L$ be its limit distribution.
Suppose that $c = \lim_{x \downarrow 0} k(x)$ is in $(0,\infty)$,
and define
\begin{equation}
K(x) = \exp \left(\int_x^1 \left(c - k(y)\right) \frac{dy}{y}\right).
\end{equation}
Then $K(x)$ is slowly varying at $0$ and $L$ satisfies
\begin{equation}
L(x)
\sim \frac{\kappa}{\Gamma(c)}(x - \gamma)^{c - 1}K(x - \gamma) \qquad \text{as} \qquad x \downarrow \gamma,
\end{equation}
where $\gamma = b/\lambda_0$ and
\[\kappa = \exp \left(c\int_0^1 (e^{-x} - 1) \frac{\dd x}{x} + c \int_1^\infty e^{-x} \frac{\dd x}{x} - \int_1^\infty k(x) \frac{\dd x}{x}\right).\]
\end{prop}

\begin{proof}
Note that the inequality
$c>0$
and Proposition~\ref{Cor:absolute_continuity}
imply that
$L$
is absolutely continuous.
Its support is by
Proposition~\ref{Prop:support}
equal to
$[\gamma,\infty)$
and
$L(\gamma)=0$.
The proof of \citet[Theorem~53.6]{Sato1999} for self-decomposable
distributions can now be applied without change.
\end{proof}

Recall that
$\mathrm{ID}_+$
and
$\SDplus$
denote the classes of infinitely divisible and self-decomposable
distributions on
$\Rplus$
respectively.
Let
$\CLIM$
be the class of distributions
on
$\Rplus$
that arise as
limit distributions of CBI-processes.

\begin{prop}
The class  $\CLIM$ is contained strictly between the self-decomposable and the infinitely divisible distributions on $\Rplus$, i.e.
 \[\SDplus \subsetneq \CLIM \subsetneq \mathrm{ID}‚_+.\]
\end{prop}
\begin{proof}
The inclusion $\CLIM \subset \mathrm{ID}_+$ follows from Theorem~\ref{Thm:limit_dis},
and the inclusion $\SDplus \subset \CLIM$ from Theorem~\ref{Thm:Jurek} and
the fact that each $\Rplus$-valued OU-type process (see~\eqref{Eq:OU_SDE})
is a CBI-process with
$R'_+(0) = -\lambda < 0$.
The strictness of the inclusions can be deduced from the following
facts:
\begin{itemize}
\item all distributions in $\SDplus$ are either degenerate or absolutely continuous (cf. \citet[Thorem~27.13]{Sato1999});
\item all distributions in $\CLIM$ are absolutely continuous on $\Rplus \setminus \set{b/\lambda_0}$, but some
concentrate non-zero mass at $\set{b/\lambda_0}$ (cf. Propositions~\ref{Prop:support} and \ref{Cor:absolute_continuity});
\item the class $\mathrm{ID}_+$ contains singular distributions (cf. \citet[Theorem~27.19]{Sato1999}).
\end{itemize}
\end{proof}
For a more direct proof of the fact that $\SDplus$ is \emph{strictly} included in $\CLIM$ we exhibit an example of a distribution that is in $\CLIM$ but not in $\SDplus$:

\begin{example}[CBI-process with non self-decomposable limit distribution]
\label{Ex:nonSD} In this example we consider the class of
CBI-processes $X$ given by a general subordinator $X^F$ and
spectrally positive process $X^R$ equal to a Brownian motion with
strictly negative drift. The Laplace exponent of $X^R$ is
$R(u)=-\alpha u^2+ \beta u$ with $\alpha > 0, \beta < 0$. It is easy
to check using~\eqref{Eq:Laplace} that the scale function of the
dual $\wh{X}^R$ and its derivative are
\begin{eqnarray}
\label{eq:BrowninaScale}
W(x) = \left[\exp\left(x\beta/\alpha\right)-1\right]/\beta & \text{and} &
W'(x) = \exp\left(x\beta/\alpha\right)/\alpha.
\end{eqnarray}
Theorem~\ref{thm:Main} implies that the characteristics of the limit
distribution $L$ are given by $\gamma=0$ and
\begin{equation}
\label{eq:k_function_BM_D_case}
k(x) = \ee^{x\beta/\alpha}
\left[\frac{b}{\alpha}+\frac{1}{\beta}\left(m(x,\infty)+\int_{(0,x]}\left(1-\ee^{-\xi\beta/\alpha}\right)\,m(\dd \xi)\right)\right]
-m(x,\infty)/\beta,
\end{equation}
where $b\in\Rplus$ is the drift and $m$ the L\'evy measure of the
subordinator $X^F$. Assuming in addition that $X^F$ is a compound
Poisson process with exponential jumps and setting parameters equal
to
$$
m(x,\infty)=\ee^{-x},\>b=0,\>\alpha=1/2,\>\beta=-1,
$$
formula ~\eqref{eq:k_function_BM_D_case} reduces to
$k(x)=2(\ee^{-x}-\ee^{-2x})$. Since this $k$-function is not
decreasing, the corresponding distribution $L$, which is in $\CLIM$,
cannot be in $\SDplus$.
\end{example}

Proposition~\eqref{prop:SuffCond_SD} gives sufficient conditions for a distribution
in $\CLIM$ to be self-decomposable.
\begin{prop}
\label{prop:SuffCond_SD}
Let $X$ be a CBI-process
and let $L$ be its limit distribution. Each of the following
conditions is sufficient for $L$ to be self-decomposable:
\begin{enumerate}[(a)]
\item $\mu = 0$ and $\alpha = 0$,
\item $\mu = 0$ and $m=0$,
\item $m=0$ and $W$ is concave on $(0,\infty)$.
\label{prop:numbering_c}
\end{enumerate}
Conversely, if $m = 0$ and $L$ is self-decomposable, then $W$ must
be concave on $(0,\infty)$.
\end{prop}

\begin{rem}
The monotonicity of the derivative of the scale function, which arises 
in Proposition~\ref{prop:SuffCond_SD}, also plays a role in 
other applications of scale functions
(e.g. control theory~\cite{Loeffen2008}; 
conjugate Bernstein functions and one-sided L\'evy processes~\cite{KyprianouRiver2008}). 
\end{rem}

\begin{proof}The first two conditions are rather trivial. In the first case 
$X$ is an OU-type process, and self-decomposability follows from the classical 
results of~\citet{Jurek1983, Sato1984} that we state as Theorem~\ref{Thm:Jurek}. 
In the second case $X$ has no jumps, and hence is a Feller diffusion. This process 
is well-studied, and its limit distribution is known explicitly. It is a shifted 
gamma distribution, which is always self-decomposable.
It remains to show (c) and the converse assertion. Assume that 
$m = 0$. By Theorem~\ref{thm:Main} we have $k(x) = bW'_+(x)$ in this
case. An infinitely divisible distribution is self-decomposable if
and only if it can be written as in \eqref{Eq:kfunction} with
decreasing $k$-function. Clearly $k$ is decreasing if and only if
$W'_+$ is, or equivalently if $W$ is concave on $\Rplus$.
\end{proof}


\appendix

\section{Additional Proofs for Section~\ref{sec:Preliminaries}}

\begin{proof}[Proof of Theorem~\ref{Thm:limit_dis}]‚
We first show that \eqref{Item:integral_condition}
is equivalent to \eqref{Item:limit} and that $L$ has to satisfy \eqref{Item:id} and
\eqref{Item:laplace}. Consider the three alternatives for the behavior of $R(u)$ that are outlined in Lemma~\ref{Lem:R_zeroes}. Through the Riccati equations~\eqref{Eq:Riccati_general} they imply the following behavior of $\psi(t,u)$: If $R'_+(0) > 0$ then $\lim_{t \to \infty} \psi(t,u) = u_0$ for all $t,u > 0$, if $R \equiv 0$ then $\psi(t,u) = u$ for all $u \ge 0$, and if $R'_+(0) \le 0$ but $R \not \equiv 0$ then $\lim_{t \to \infty} \psi(t,u) = 0$. Moreover,
\begin{equation}
\lim_{t \uparrow \infty} -\log \E{e^{-uX_t}} = \lim_{t \uparrow \infty}
\left(\phi(t,u)  + x \psi(t,u)\right) = \int_0^\infty{F(\psi(r,u))\,\dd r}  + x \cdot \lim_{t \to \infty} \psi(t,u).
\end{equation}
We see that if $R'_+(0) > 0$ or $R \equiv 0$ the right-hand side diverges for $u > 0$, and hence no limit distribution exists in these cases. In case that $R'_+(0) \le 0$ and $R \not \equiv 0$, the transformation $s = \psi(r,u)$ yields that
\begin{equation}\label{Eq:limit_dis_interm}
\lim_{t \uparrow \infty} -\log \E{e^{-uX_t}} = \int_0^u \frac{F(s)}{R(s)}\dd s.
\end{equation}
This integral is finite, if and only if condition \eqref{Eq:integral_condition} holds. If it is finite then L\'evy's continuity theorem for Laplace transforms guarantees the existence of, and convergence to, the limit distribution $L$ with Laplace
exponent given by \eqref{Eq:Gen_Jurek_Vervaat}. It is also clear that $L$ must be infinitely divisible, since it is the limit of infinitely divisible distributions. If on the other hand the
integral in \eqref{Eq:limit_dis_interm} is infinite for some $u
\in \Rplus$, then there is no pointwise convergence of Laplace
transforms, and hence also no weak convergence of $X_t$ as $t \to
\infty$.

To complete the proof it remains to show that any limit distribution
is also invariant and vice versa, i.e. that \eqref{Item:invariant}
is equivalent to \eqref{Item:limit}. Assume that $\wt{L}$ is an
invariant distribution of $\Xt$, and has Laplace exponent $\wt{l}(u)
= -\log \int_{[0,\infty)} e^{-ux} \dd \wt{L}(x)$.
Denote
$f_u(x) = e^{-ux}$
and note that the invariance of
$\wt{L}$
implies
\begin{equation}
\int_{[0,\infty)} f_u(x) \dd \wt{L}(x) = \int_{[0,\infty)} P_t f_u(x) \dd
\wt{L}(x) = e^{-\phi(t,u)} \int_{[0,\infty)} e^{-x \psi(t,u)} \dd \wt L(x)
\end{equation}
for all $t,u \ge 0$. This can be rewritten as
$\wt l(u) = \phi(t,u) + \wt{l}(\psi(t,u)).$
Taking derivatives with respect to $t$ and evaluating at $t = 0$
this becomes
$0 = F(u) + \wt{l}'(u)R(u).$
Since $\wt{l}(u)$ is continuous on $\Rplus$ with $\wt{l}(0) = 0$,
the above equation can be integrated to yield $\wt{l}(u) =
-\int_0^u{\frac{F(s)}{R(s)}\dd s}$. By the first part of the proof
this implies that a limit distribution $L$ exists, with Laplace
exponent $l(u)$ coinciding with $\wt l(u)$. We conclude that also
the probability laws $L$ and $\wt{L}$ on $\Rplus$ coincide, i.e. $L
= \wt L$ . Conversely, assume that a limit distribution $L$ exists.
To show that $L$ is also invariant, note that~\eqref{Eq:CBI_def},
\eqref{Eq:Riccati_general} and \eqref{Eq:Gen_Jurek_Vervaat} imply
\begin{align*}
\int_{[0,\infty)} P_t f_u(x) \dd L(x) &= \exp\left(-\phi(t,u) -
l(\psi(t,u))\right) = \exp\left(-\int_0^t F(\psi(r,u))\dd r + \int_0
^{\psi(t,u)} \frac{F(s)}{R(s)}\dd s\right)  \\
&= \exp\left(\int_0^{u} \frac{F(s)}{R(s)}\dd s\right) =
\int_{[0,\infty)} f_u(x) \dd L(x).
\end{align*}
This completes the proof.
\end{proof}

\begin{proof}[Proof of Corollary~\ref{Cor:log_moment}]
Assume that $\int_{[1,\infty)} \log \xi \, m(\dd \xi) < \infty$. From the concavity of $R(u)$, Lemma~\ref{Lem:R_zeroes}
and the fact that
$F(u) \ge 0$ for all $u \ge 0$ we obtain that
\begin{equation}
\label{Eq:integral_bound}
 0 \le - \int_0^u{\frac{F(s)}{R(s)}}\,\dd s  \le -\frac{1}{R'(0)}\int_0^u{\frac{F(s)}{s}\,\dd s} 
= -\frac{1}{R'(0)} \left(bu + \int_0^u \int_{(0,\infty)} \frac{1 -
e^{-s\xi}}{s} m(\dd \xi) \dd s\right).
\end{equation}
In order to show that this upper bound is finite, it is enough to
show that the double integral on the right takes a finite value.
Since the integrand is positive, the integrals can be exchanged by
the Tonelli-Fubini theorem. Defining the function $M(\xi) = \int_0^u
\frac{1 - e^{-s\xi}}{s} \dd s$, we can write
\[\int_0^u \int_{(0,\infty)} \frac{1 -
e^{-s\xi}}{s} m(\dd\xi) \dd s = \int_{(0,\infty)}M(\xi)\,m(\dd\xi).\] An
application of L'H\^opital's formula reveals the following boundary
behavior of $M(\xi)$:
\begin{equation}
\label{eq:limits_for_M}
\lim_{\xi \to 0}\frac{M(\xi)}{\xi} = u \qquad \text{and} \qquad \lim_{\xi \to
\infty}\frac{M(\xi)}{\log \xi} = 1.
\end{equation}
Choosing suitable constants $C_1, C_2>0$ we can bound $M(\xi)$ from above by
$C_1 \xi$ on $(0,1)$ and by $C_2 \log \xi$ on $[1,\infty)$. Note that
$m(\dd \xi)$ integrates the function $\xi \mapsto C_1 \xi$ on $(0,1)$ by
Theorem~\ref{Thm:KW}, and integrates the function $\xi
\mapsto C_2 \log \xi$ on $[1,\infty)$ by assumption. Hence
\[\int_{(0,\infty)}M(\xi)m(\dd \xi) \le
C_1 \int_{(0,1)} \xi\, m(\dd \xi) + C_2 \int_{[1,\infty)} \log \xi\,
m(d\xi) < \infty,\] and we have shown that the upper bound in
\eqref{Eq:integral_bound} is finite and that \eqref{Eq:integral_condition} holds true.

Suppose now that $\int_{\xi
> 1} \log \xi \, m(\dd \xi) = \infty$. Since $R'_+(0) < 0$ we can find $\epsilon, \delta > 0$ such that $R'_+(0) + \epsilon < 0$ and $R(u) \ge (R'_+(0) - \epsilon)u$ for all $u \in (0,\delta)$. Hence,
\begin{align}
- \int_0^u{\frac{F(s)}{R(s)}}\,\dd s \ge \frac{1}{\epsilon - R'_+(0)}\int_0^u{\frac{F(s)}{s}\, \dd s}
=  \frac{1}{\epsilon - R'_+(0)} \left(bu + \int_0^u
\int_{(0,\infty)} \frac{1 - e^{-s\xi}}{s} m(\dd \xi)
\dd s\right),\label{Eq:integral_bound2}
\end{align}
for all $u \in (0,\delta)$. Exchanging integrals by the Tonelli-Fubini theorem and using the
function $M(\xi)$ defined above we get
\[\int_0^u \int_{(0,\infty)} \frac{1 -
e^{-s\xi}}{s} m(\dd \xi) ds = \int_{(0,\infty)}M(\xi)m(\dd \xi)
\ge C_2' \int_{[1,\infty)} \log \xi \, m(\dd \xi) = \infty,\]
where
$C_2'>0$
is a finite constant which exists by the second limit in~\eqref{eq:limits_for_M}.
This shows that the right hand side of \eqref{Eq:integral_bound2} is
infinite and hence that \eqref{Eq:integral_condition} can not hold true.
\end{proof}

\bibliographystyle{plainnat}
\bibliography{references}

\begin{thebibliography}{21}
\providecommand{\natexlab}[1]{#1}
\providecommand{\url}[1]{\texttt{#1}}
\expandafter\ifx\csname urlstyle\endcsname\relax
  \providecommand{\doi}[1]{doi: #1}\else
  \providecommand{\doi}{doi: \begingroup \urlstyle{rm}\Url}\fi

\bibitem[Bertoin(1996)]{Bertoin1996}
Jean Bertoin.
\newblock \emph{{L}{\'e}vy processes}.
\newblock Cambridge University Press, 1996.

\bibitem[Caballero et~al.(2010)Caballero, {P{\'e}rez Garmendia}, and {Uribe
  Bravo}]{Caballero2010}
Ma.~Emilia Caballero, Jos{\'e}~Luis {P{\'e}rez Garmendia}, and Ger{\`o}nimo
  {Uribe Bravo}.
\newblock A {L}amperti type representation of continuous-state branching
  processes with immigration.
\newblock arXiv:1012.2346, 2010.

\bibitem[Chan et~al.(2011)Chan, Kyprianou, and Savov]{Chan_et_al2011}
T.~Chan, A.~E. Kyprianou, and M.~Savov.
\newblock Smoothness of scale functions for spectrally negative {L}\'evy
  processes.
\newblock \emph{Probability Theory and Related Fields}, 150\penalty0
  (3-4):\penalty0 691--708, 2011.

\bibitem[{\c{C}}inlar and Pinsky(1971)]{Cinlar1971}
E.~{\c{C}}inlar and M.~Pinsky.
\newblock A stochastic integral in storage theory.
\newblock \emph{Zeitschrift f{\"{u}}r Wahrscheinlichkeitstheorie und verwandte
  Gebiete}, 17:\penalty0 227--240, 1971.

\bibitem[Dawson and Li(2006)]{Dawson2006}
D.~A. Dawson and Zenghu Li.
\newblock Skew convolution semigroups and affine {M}arkov processes.
\newblock \emph{The Annals of Probability}, 34\penalty0 (3):\penalty0 1103 --
  1142, 2006.

\bibitem[Heathcote(1965)]{Heathcote1965}
C.~R. Heathcote.
\newblock A branching process allowing immigration.
\newblock \emph{Journal of the Royal Statistical Society B}, 27\penalty0
  (1):\penalty0 138--143, 1965.

\bibitem[Heathcote(1966)]{Heathcote1966}
C.~R. Heathcote.
\newblock Corrections and comments on the paper `{A} branching process allowing
  immigration'.
\newblock \emph{Journal of the Royal Statistical Society B}, 28\penalty0
  (1):\penalty0 213--217, 1966.

\bibitem[James et~al.(2008)James, Roynette, and Yor]{James2008}
Lancelot~F. James, Bernard Roynette, and Mark Yor.
\newblock Generalized gamma convolutions, {D}irichlet means, {T}horin measures,
  with explicit examples.
\newblock \emph{Probability Surveys}, 5:\penalty0 346--415, 2008.

\bibitem[Jurek and Vervaat(1983)]{Jurek1983}
Z.~J. Jurek and W.~Vervaat.
\newblock An integral representation for self-decomposable {B}anach space
  valued random variables.
\newblock \emph{Zeitschrift f{\"{u}}r Wahrscheinlichkeitstheorie und verwandte
  Gebiete}, 62:\penalty0 247--262, 1983.

\bibitem[Kawazu and Watanabe(1971)]{Kawazu1971}
Kiyoshi Kawazu and Shinzo Watanabe.
\newblock Branching processes with immigration and related limit theorems.
\newblock \emph{Theory of Probability and its Applications}, XVI\penalty0
  (1):\penalty0 36--54, 1971.

\bibitem[Keller-Ressel and Steiner(2008)]{KS2008}
Martin Keller-Ressel and Thomas Steiner.
\newblock Yield curve shapes and the asymptotic short rate distribution in
  affine one-factor models.
\newblock \emph{Finance and Stochastics}, 12\penalty0 (2):\penalty0 149 -- 172,
  2008.

\bibitem[Kyprianou and Rivero(2008)]{KyprianouRiver2008}
A.E. Kyprianou and V.~Rivero.
\newblock Conjugate and complete scale functions for spectrally negative
  {L}\'evy processes.
\newblock \emph{Electronic Journal of Probability}, 13\penalty0 (57):\penalty0
  1672--1701, 2008.

\bibitem[Kyprianou(2006)]{Kyprianou2006}
Andreas~E. Kyprianou.
\newblock \emph{Introductory Lectures on Fluctuations of L{\'e}vy Processes
  with Applications}.
\newblock Springer, 2006.

\bibitem[Lamperti(1967)]{Lamperti1967}
John Lamperti.
\newblock Continuous-state branching processes.
\newblock \emph{Bulletin of the AMS}, 73:\penalty0 382--386, 1967.

\bibitem[Li(2011)]{Li2011}
Zenghu Li.
\newblock \emph{Measure-Valued Branching {M}arkov Processes}.
\newblock Probability and its Applications. Springer, 2011.

\bibitem[Loeffen(2008)]{Loeffen2008}
Ronnie Loeffen.
\newblock On optimality of the barrier strategy in de {F}inetti's dividend
  problem for spectrally negative {L}\'evy processes.
\newblock \emph{Annals of Applied Probability}, 18\penalty0 (5):\penalty0
  1669--1680, 2008.

\bibitem[Masuda and Yoshida(2005)]{Masuda2005}
H.~Masuda and N.~Yoshida.
\newblock Asymptotic expansion for {B}arndorff-{N}ielsen and {S}hephard{'}s
  stochastic volatility model.
\newblock \emph{Stochastic Processes and their Applications}, 115:\penalty0
  1167--1186, 2005.

\bibitem[Pinsky(1972)]{Pinsky1972}
Mark Pinsky.
\newblock Limit theorems for continuous state branching processes with
  immigration.
\newblock \emph{Bulletin of the AMS}, 78\penalty0 (2):\penalty0 242--244, 1972.

\bibitem[Sato(1999)]{Sato1999}
Ken-Iti Sato.
\newblock \emph{{L}{\'e}vy processes and infinitely divisible distributions}.
\newblock Cambridge University Press, 1999.

\bibitem[Sato and Yamazato(1984)]{Sato1984}
{Ken-iti} Sato and M.~Yamazato.
\newblock Operator-selfdecomposable distributions as limit distributions of
  processes of {O}rnstein-{U}hlenbeck type.
\newblock \emph{Stochastic Processes and Applications}, 17:\penalty0 73--100,
  1984.

\bibitem[Yamazato(1978)]{Yamazato1978}
M.~Yamazato.
\newblock Unimodality of infinitely divisible distribution functions of class
  {L}.
\newblock \emph{{A}nnals of {P}robability}, 6:\penalty0 523--531, 1978.

\end{thebibliography}

\end{document}